\documentclass[12pt]{amsart}
\usepackage[T1]{fontenc}	    
\usepackage[utf8]{inputenc}         
\usepackage[english]{babel}
\usepackage{amscd}
\usepackage{amsmath,amsfonts,amsthm,epsfig,latexsym,graphicx,amssymb,psfrag}
\usepackage{xspace}
\usepackage{xcolor}
\usepackage{enumerate}
\usepackage{array} 
\usepackage{appendix}

\newtheorem{coro}{Corollary}

\newtheorem{defi}{Definition }[section]

\newtheorem{lemm}{Lemma}[section]
\newtheorem{prop}{Proposition}[section]

\newtheorem{theo}{Theorem} \newtheorem*{theor}{Theorem C12.14 of [BS16]}

\theoremstyle{definition}
\newtheorem{rema}{Remark}

\def\C{C \hskip -3pt}
\newfont{\df}{cmssbx10} 

\newtheorem{definition}{Definition }
\newtheorem*{definition*}{Definition }
\newtheorem{lemma}[definition]{Lemma}
\newtheorem{proposition}[definition]{Proposition }

\newtheorem{remark}[definition]{Remark }

\newtheorem*{notation*}{Notation}

\newcommand{\ds}{\displaystyle}

\newcommand{\Q}{\mathbb{Q}}
\newcommand{\R}{\mathbb{R}}

\newcommand{\iet}{\mathcal G^+(J)}
\newcommand{\ietf}{\mathcal G(J)}
\newcommand{\ieta}{\mathcal A^+(J)}
\newcommand{\ietaa}{\mathcal A^+([0,1))}
\newcommand{\ietg}{\mathcal{PC}^+(J)}
\newcommand{\ietgg}{\mathcal{PC}^+([0,1))}
\newcommand{\supp}{\text{Supp}}
\newcommand{\disc}{\text{Disc}}
\newcommand{\fix}{\text{Fix}}
\newcommand{\homeo}{\text{Homeo}}
\newcommand{\diff}{\text{Diff }}
\newcommand{\saf}{\textsc{saf}} 
\newcommand{\id}{\text{Id}}
\newcommand{\TS}{T_{\{n_i\}}}
\newcommand{\FS}{F_{\{n_i\}}}
\newcommand{\DDT}{T^{\ds\text{"}}\hskip-7pt _{\{n_i\}}} 

\title[Uniform simplicity of $\mathcal{PC}(I)$]{Uniform simplicity for subgroups of piecewise continuous bijections of the unit interval.}

\begin{document}

\author{Nancy Guelman {\tiny and} Isabelle Liousse {\tiny with an appendix by} Pierre ARNOUX}
\address{{\bf  Nancy  GUELMAN}, IMERL, Facultad de Ingenier\'{\i}a, {Universidad de la Rep\'ublica, C.C. 30, Montevideo, Uruguay.} \emph {nguelman@fing.edu.uy}}

\address{{\bf Isabelle LIOUSSE}, Univ. Lille, CNRS, UMR 8524 - Laboratoire Paul Painlevé, F-59000 Lille, France. \emph {isabelle.liousse@univ-lille.fr}}

\address{{\bf Pierre ARNOUX}, Département de Mathématiques, Université d'Aix-Marseille Campus de Luminy,
 Avenue de Luminy - Case 907, 13288 Marseille Cedex 9, France. \emph{pierre@pierrearnoux.fr}}

\begin{abstract} Let $I=[0,1)$ and $\mathcal{PC}(I)$ [resp. $\mathcal{PC}^+(I)$] be the quotient group of the group of all piecewise continuous [resp.  piecewise continuous and orientation preserving] bijections of $I$ by its normal subgroup consisting in elements with finite support (i.e. that are trivial except at possibly finitely many points). {Unpublished Theorems of Arnoux (\cite{Ar1}) state that $\mathcal{PC}^+(I)$ and certain groups of interval exchanges are simple, their proofs are the purpose of the Appendix.} Dealing with piecewise direct affine maps, we prove the simplicity of the group $\mathcal A^+(I)$ (see Definition \ref{defAIET}). These results can be improved. Indeed, a group $G$ is uniformly simple if there exists a positive integer $N$ such that for any $f,\phi \in G\setminus\{\id\}$, the element $\phi$ can be written as a product of at most $N$ conjugates of $f$ or $f^{-1}$.

We provide conditions which guarantee that a subgroup $G$ of $\mathcal{PC}(I)$ is uniformly simple. As Corollaries, we obtain that $\mathcal{PC}(I)$, $\mathcal{PC}^+(I)$, $PL^+ (\mathbb S^1)$, $\mathcal A(I)$, $\mathcal A^+(I)$ and some Thompson like groups included the Thompson group $T$ are uniformly simple.
\end{abstract}
\maketitle 

\section{Introduction}

The algebraic study of groups consisting in continuous transformations of a topological space was initiated  by Schreier and Ulam in 1934 (\cite{SU}) and the question of the simplicity of such groups was raised.

\smallskip

\begin{defi} A group $G$ is \textbf{simple} if any normal subgroup of $G$ is either trivial or equal to $G$. In particular, a simple group $G$ is \textbf{perfect} that is $G$ coincides with the normal subgroup generated its commutators $[a,b]=aba^{-1}b^{-1}$ with $a,b \in G$. 
\end{defi}
 
\smallskip
 
In \cite{UV}, Ulam and Von Neuman proved that the identity component in the group of homeomorphisms of the circle or the $2$-sphere is a simple group. In the seventies lots of (smooth) transformation groups were studied by Epstein, Herman, Thurston, Mather, Banyaga, and proved to be simple (see the books \cite{Ba} or \cite{Bo}).

\smallskip

In \cite{Ula}, Ulam explained that \cite{UV} establishes a sharper theorem: \textit{"for every $f$ and $\phi$ non-trivial and isotopic to identity homeomorphisms of the circle or the $2$-sphere, there exists a fixed number $N$ of conjugates of $f$ or $f^{-1}$ whose product is $\phi$"}. This number does not exceed $23$ and Ulam raised the question of finding the optimal bound. The issue was taken up again in updated versions of the Scottish book (see \cite{ScB}, Problem 29) in relation with Nunnally's work (\cite{Nu}) which states that $N$ is less than $3$ for certain groups of homeomorphisms. This leads to the following

\begin{defi}  Let $N$ be a positive integer.

A group $G$ is $\boldsymbol{N-}$\textbf{uniformly perfect} if any product of commutators in $G$ can be written as a product of {at most} $N$ commutators in $G$.

A group $G$ is $\boldsymbol{N-}$\textbf{uniformly simple} if for any pair $\{f,\phi\}$ of non trivial elements of $G$, one can express $\phi$ as a product of {at most} $N$ conjugates of $f$ or $f^{-1}$ in $G$.
\end{defi}

\smallskip

Note that uniform simplicity implies simplicity but uniform perfectness does not imply perfectness. 

\smallskip

In this context, Nunnally's work establishes the 3-simplicity for certain groups of homeomorphisms. But Nunnally's techniques fail when requiring groups to preserve additional structures (e.g. smooth, PL or area).
Tsuboi (\cite{Ts}) showed the uniform simplicity of the identity component  $\textit{\diff}^r(M^n)_0$ of the group of $C^r$-diffeomorphisms ($1\leq r \leq \infty$, $r\not=n+1$) of a compact connected $n$-dimensional manifold $M^n$ with handle decomposition without handles of index $\frac{n}{2}$. As a Corollary and under the same assumption on $r$, he obtained that 
{\textit{$\textit{\diff}^r(\mathbb S^n)_0$ is $12$-uniformly simple}}.

\smallskip
 
Given a group $G$ and a non trivial subset $S$ of $G$ which is closed under inversion and conjugation, if $G$ is $N$-uniformly simple then any $g\in G$ can be expressed as a product of at most $N$ elements of $S$. In particular, $S$ can be the set consisting of involutions, finite order elements, commutators or reversible maps. Recall that $g\in G$ is said to be \textbf{reversible} if $g$ is conjugated in $G$ to its inverse. In the O'Farell and Short survey on reversibility (\cite{FS} p35), the authors raised the related questions: "\textit{Given $G$ a group, does there exist a positive integer $n$ [resp. $m$] such that  $G$ coincides with $\mathcal I^n=\{ \ h_1 \ \cdots \ h_n  \text{ with } h_i^2=\id  \}$ [resp. with $\mathcal R^m= \{\ r_1 \ \cdots \ r_m  \text{ with } r_i  \text{ reversible}  \}$]~?}" Clearly, for uniformly simple groups both questions have a positive answer.

\smallskip

In this paper, we do not further assume that transformations are continuous and we focus on dimension one. The groups we are interested in are described by the following 

\begin{defi} Let $I=[0,1)$ be the unit interval. 
\begin{itemize}
\item A \textbf{piecewise continuous bijection of $I$} is a bijection $f$ of $I$ that is continuous outside a finite subset of $I$ called {\textbf{discontinuity set} and denoted by  \textbf{\disc}$\boldsymbol{(f)}$. The \textbf{support} of $f$ is the set \textbf{\supp}$\boldsymbol{(f)}=\{ x\in I\ \vert \ f(x)\not=x \}$.}
\item Let $\boldsymbol{\widehat{\mathcal{PC}}(I)}$ be the group of piecewise continuous bijections of $I$. We denote by $\boldsymbol{\mathcal{PC}(I)}$ the quotient group of $\widehat{\mathcal{PC}}(I)$ by its normal subgroup consisting in elements with finite support and the subgroup of $\mathcal{PC}(I)$ consisting in classes of piecewise increasing elements is {referred as} $\boldsymbol{\mathcal{PC}^+(I)}$.
\end{itemize}
\end{defi}

By taking the unique right continuous representative for all $f$ in $\mathcal{PC}^+(I)$, the group $\mathcal{PC}^+(I)$ can be identified with the group of right continuous and piecewise increasing bijections of $I$. But such a representative may not exist for some elements of $\mathcal{PC}(I)$. 

\begin{defi} Let $f\in \mathcal{PC}(I)$. We say that a representative of $f$ is \textbf{good} if it minimizes the number of discontinuity points among the elements of the class $f$.
\end{defi}

Note that this minimizing condition does not guarantee uniqueness, but all the good representatives of a given element of $\mathcal{PC}(I)$ have the same discontinuity point set, the same image of the discontinuity point set and they coincide on their common continuity set. However, it is possible to require more properties in order to exhibit "canonical" representatives.

More precisely, let $f\in \mathcal{PC}(I)$ and $\tilde{f}$ be a good representative of $f$ with discontinuity points $a_i$ where  $0=a_1< \cdots <a_n<1$.  We consider $\sigma$ the finitely supported bijection which sends $\tilde{f}(a_i)$ to $b_j$ the left endpoint of $\tilde{f}((a_i,a_{i+1}))$ with the convention that $a_{n+1}=1$. Note that $\sigma$ is well-defined since the set of all $\tilde{f}(a_i)$ is equal to the set of all $b_j$.

The map $\sigma \tilde{f}$ is a good representative of $f$ and it satisfies $\sigma \tilde{f}([a_i, a_{i+1}))= [b_j, b_{j+1})$,
with the convention that $b_{n+1}=1$. Clearly $\sigma \tilde{f}$ is the unique good representative of $f$ that has this property. Then we give the following

\begin{defi} Let $f\in \mathcal{PC}(I)$. We define the \textbf{best} representative $\boldsymbol{\widehat{f}}$ of $f$  to be the unique good representative of $f$ such that $\hat f([a_i, a_{i+1}))$ is a right-open and left-closed interval, where $a_i$, $1\leq i \leq n$ are the discontinuity points of $\hat f$ and  $a_{n+1}=1$.
\end{defi}

\hskip 1.2 truecm
\includegraphics[width=12cm, height=5cm]{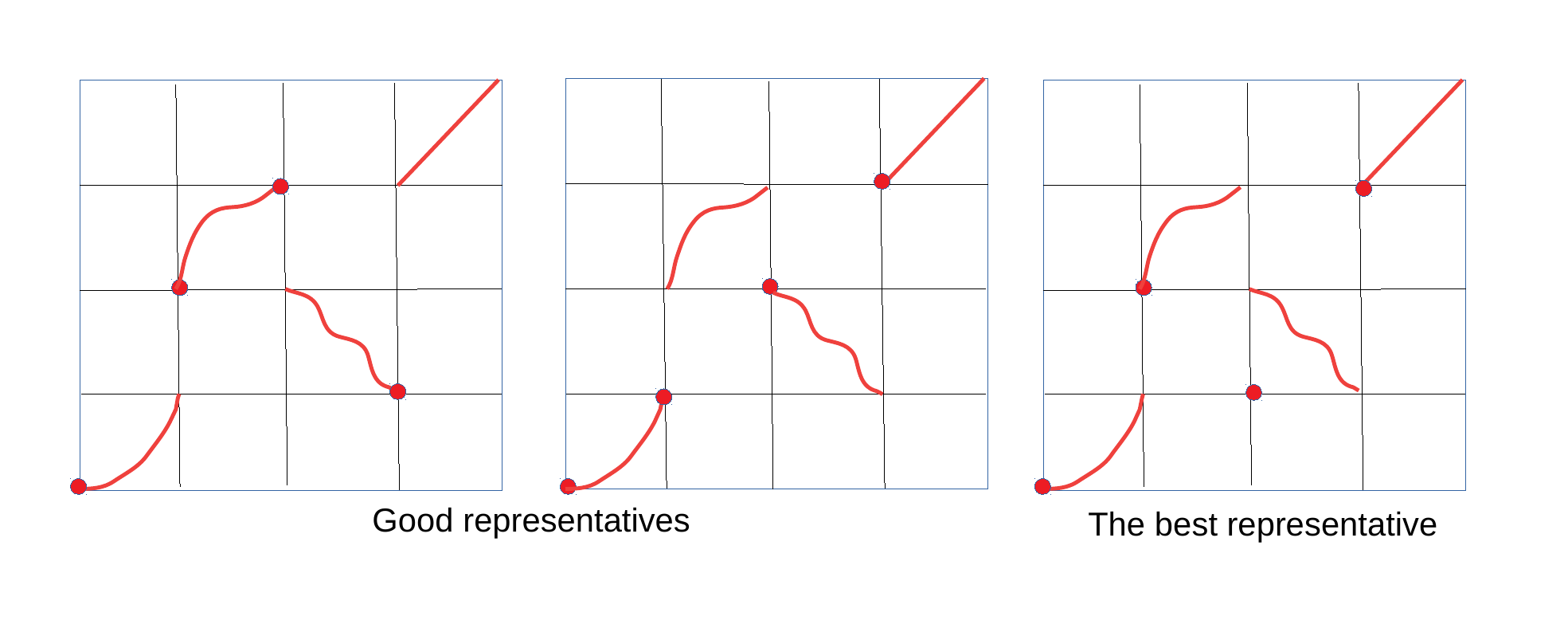} 
\vskip - 0.6 truecm

\smallskip

\begin{rema} \ 
\begin{itemize} 
\item If $f\in\mathcal{PC}^+(I)$ then $\widehat{f}$ is the right-continuous representative of $f$. More generally, for $f\in\mathcal{PC} (I)$, $\widehat{f}$ is the good representative of $f$ that is right continuous at the left endpoints of the continuity intervals where $\widehat{f}$ is orientation preserving and for the continuity intervals where $\widehat{f}$ is orientation reversing,  $\widehat{f}$ sends their left endpoints to the left endpoints of their images.
\item Note that the map $f \mapsto \widehat{f}$ is not a morphism (i.e. there exist $f$ and $g$ such that $\widehat{f \circ g} \not= \widehat{f} \circ \widehat{g}$). 
\end{itemize} 
\end{rema}

\medskip

Since the maps we deal with, are only piecewise continuous, the interval $[0,1)$ can be identified with the unit circle $\mathbb S^1$ and it is equivalent to consider a piecewise continuous bijection as a map $\mathbb S^1 \rightarrow \mathbb S^1$ (see \cite{Cor}). We refer as "continuous versions" of a subgroup $G$ of $\mathcal{PC}(I)$ the subgroups of $G$ consisting in classes of continuous elements of either the interval or the circle. The continuous versions of $\mathcal{PC}^+(I)$ are 
$\homeo^+ (I)$ and $\homeo^+ (\mathbb S^1)$ and their simplicity was shown by Epstein (\cite{Ep}). 

\smallskip

Arnoux (\cite{Ar1}) proved that $\mathcal{PC}^+(I)$ and certain groups of interval exchanges, right after defined, are simple. Unfortunately, these works are unpublished and we express our gratitude to P. Arnoux for reproducing and joining them as an appendix.

\medskip

\begin{defi}\label{defAIET} Let $f\in \widehat{\mathcal{PC}}(I)$. 

\smallskip

$\bullet$ The map $f$ is an \textbf{affine interval exchange transformation (AIET)} if there exists a finite subdivision $0=a_1<a_2<\dots<a_p<a_{p+1}=1$ of \ $[0,1)$ such that for any $i=1,\cdots,p$
$$f_{\vert [a_i,a_{i+1})} (x)  = \lambda_i x + \beta_i, \qquad \lambda_i \in \mathbb R^+, \ \beta_i\in \mathbb R.$$
We define the group $\boldsymbol{\mathcal A^+(I)}$ to be the set of all AIET.

\smallskip

$\bullet$ The map $f$ is an \textbf{affine interval exchange transformation with flips (FAIET)} if there is a finite subdivision  $0=a_1<a_2<\dots<a_p<a_{p+1}=1$ of $[0,1)$ such that for any $i=1, \cdots, p$
$$f_{\vert (a_i,a_{i+1}) } (x)  = \lambda_i x + \beta_i, \qquad \lambda_i \in \mathbb R, \ \beta_i\in \mathbb R.$$
\noindent The numbers $\lambda_i$ are called the \textbf{slopes of~$f$} and their set is denoted by $\boldsymbol{\Lambda(f)}$.

\smallskip

We denote by $\boldsymbol{\widehat{\mathcal A(I)}}$ the group of all FAIET and we define $\boldsymbol{\mathcal A(I)}$ to be the group of all classes of FAIET.

\medskip

$\bullet$ An \textbf{interval exchange transformation (IET)} is $f\in \mathcal A^+(I)$ with $\Lambda(f)=\{1\}$. We define the group $\boldsymbol{\mathcal G^+(I)}$ to be the set of all IET.

$\bullet$ An \textbf{interval exchange transformation with flips (FIET)} is $f\in \widehat{\mathcal A(I)}$ with $\Lambda(f)\subset \{1, -1\}$. We define the group $\boldsymbol{\mathcal G(I)}$ to be the set of all classes of FIET.
\end{defi}

The continuous versions of $\mathcal A^+(I)$ are $PL^+(I)$ and $PL^+(\mathbb S^1)$, the groups of piecewise affine homeomorphisms ({commonly referred as \textbf{PL homeomorphisms}}) of the unit interval and the circle respectively. 

\medskip

For interval exchange transformations, Arnoux (\cite{Ar1}, \cite{Ar}) and Sah (\cite{Sah}) established that \textit{$[\mathcal G^+(I),\mathcal G^+ (I)]$ is simple} and an unpublished part of \cite{Ar1} (III Proposition 1.4) showed that \textit{$\mathcal G(I)$ is simple.} In \cite{Ep}, Epstein proved that \textit{$PL^+ (\mathbb S^1)$ and $[PL^+ (I),PL^+ (I)]$ are simple}. In Section \ref{appendix1}, we prove 

\begin{theo} \label{thA} The group $\mathcal A^+(I)$ is simple.\end{theo}

It was not proven in \cite{Ar1} that $\mathcal A^+(I)$ is simple, however the tools of \cite{Ar1} provide a different proof which is detailed in the appendix. Recently, Lacourte (\cite{Lac} Theorem 1.4) proved that \textit{$\mathcal{PC}(I)$ and $\mathcal A(I)$ are simple}. Note that groups of piecewise affine bijections are particularly known because of the popularity of Thompson’s groups and their generalizations.

\medskip

\begin{defi} Let $\Lambda \subset \mathbb R^{+*}$ be a multiplicative subgroup and $A \subset \mathbb R$ be an additive subgroup which is closed under multiplication by $\Lambda $ and such that $1\in A$.

\smallskip

$\bullet$ The \textbf{Bieri-Strebel groups} are:
\begin{itemize}
\item[--] $V_{\Lambda, A}$ the subgroup of $\mathcal A^+ (I)$ consisting of elements with slopes in $\Lambda$, discontinuity points and their images in $A$, 
\item[--] $T_{\Lambda, A}$ the intersection subgroup of $PL^+ (\mathbb S^1)$ with $V_{\Lambda, A}$ and
\item[--] $F_{\Lambda, A}$ the intersection subgroup of $PL^+ (I)$ with $V_{\Lambda, A}$.
\end{itemize} 

\smallskip

$\bullet$ In the case that $A=\mathbb Z [{1}/{m}]$ and $\Lambda=\langle m \rangle$ with $m\in \mathbb N^*$, we get the \textbf{Higman-Thompsom group} $V_m$ and the \textbf{Brown-Thompson groups} $T_{m}$ and $F_{m}$.

\medskip

$\bullet$ Let $1<n_1<\dots<n_p$ be $p$ integers generating a rank $p$ free abelian multiplicative subgroup $\Lambda=\langle n_i\rangle \subset\mathbb Q^{+*}$. The \textbf{Stein-Thompson groups} are $T_{\Lambda,A}$ and $F_{\Lambda, A}$ with $A=\mathbb Z [\Lambda]$. They are denoted by $\TS$ and $\FS$.
\end{defi}

\medskip

It was shown by Thompson that \textit{$T_2$ and $V_2$ are simple} (see e.g. \cite{CFP}). Generalizing a result of Brown (\cite{Bro}), Stein (\cite{Ste}) proved that \textit{$\TS$ and $\FS$ are finitely presented and $\DDT$ is simple}. In Section \ref{appendix2}, we prove 

\begin{theo} \label{ThStTh} The Stein-Thompson groups $T_{\{n_1,n_1^ {2k}-1,... , n_p\}}$ are simple. \end{theo}

\medskip

From now on, we focus on uniform simplicity. For the group $[PL^+(I),PL^+(I)]$, it has been implicitly proved by Burago and Ivanov (\cite{BI}). 

Cornulier communicated us that $[\mathcal{G}^+(I),\mathcal{G}^+(I)]$ and ${\mathcal{G}}(I)$ are not uniformly simple. Indeed, if the support of an IET or FIET $f$ has length less than $\frac 1 N$ then any product of $N$ conjugates of $f$ or $f^{-1}$ can not have full support. However, in \cite{GLFIET}, we prove that ${\mathcal{G}}(I)$ is $6$-perfect.

\bigskip

Before stating our main result, we give necessary related notions.

\smallskip

\begin{defi}  Let $a \in [0,1)$. 

\smallskip

\begin{itemize}
\item Let $\delta>0$, we set $\boldsymbol{V_{\delta}(0)}=[0, \delta) \cup (1-\delta, 1)$ and $\boldsymbol{V_{\delta}(a)}=(a- \delta, a+\delta)$, for $\delta$ small enough.

\smallskip

\item When identifying $[0,1)$ with $\mathbb S^1$, an arc contained in $\mathbb S^1 \setminus V_{\delta}(a)$ for some positive $\delta$ is referred as $\boldsymbol{a-}$\textbf{proper interval}. More formally: 
{ \footnotesize\begin{enumerate}[$\boldsymbol{-}$]
\item A $0$-proper interval is an interval having endpoints $\{c,d\}$ with $0<c<d<1$.
\item For $a\in (0,1)$, an $a$-proper interval is either an interval with endpoints $\{c_1,c_2\}$ with $0<c_1<c_2<a$ or $\{d_1,d_2\}$ with $a<d_1<d_2<1$ or the disjoint union of a left-closed interval with endpoints $\{0,c\}$ and an interval with endpoints $\{d,1\}$ with $0<c<a<d<1$.
\end{enumerate} }

\smallskip

\item Let $g\in \widehat{\mathcal{PC}}(I)$, the \textbf{fix point set} of $g$ is the set \textbf{\fix}$\boldsymbol{(g)}=\{ x\in I \ \vert \ g(x)=x \}$.

\smallskip

\item We denote $\ds \boldsymbol{B\widehat{\mathcal{PC}} (I)_a}=\{g\in \widehat{\mathcal{PC}}(I) \ \vert \ \exists \delta>0 : \ V_{\delta}(a)\subset \fix(g)\}$ and we define $\boldsymbol{B{\mathcal{PC}} (I)_a}$ to be the image of $B\widehat{\mathcal{PC}}(I)_a$ in $\mathcal{PC} (I)$ by the quotient morphism. 
\end{itemize}

\smallskip

Let $G$ be a subgroup of $\mathcal{PC}(I)$.
\begin{itemize}
\item We set $\boldsymbol{BG_a} = G \cap B{\mathcal{PC}} (I)_a$.
\item The \textbf{regular $\boldsymbol{G}$-orbit} of $a$ is the set $\boldsymbol{G_{reg}(a)}$ consisting of points $x\in I$ for which there exists $g\in G$ such that $\widehat{g}$ is continuous at $a$ and $\widehat{g}(a)=x$, with the convention that $\widehat{g}$ is continuous at $0$ if $\lim\limits_{x\rightarrow 0+}\widehat{g}(x)=\lim\limits_{x\rightarrow 1-}\widehat{g}(x)$.
\end{itemize}
\end{defi}

\begin{rema}\label{remBGa} 
Note that $f\in B{\mathcal{PC}} (I)_a$ if and only if its best representative $\widehat{f}\in B\widehat{\mathcal{PC}} (I)_a$. 
\end{rema}

Now we introduce the conditions that will guarantee that a perfect subgroup of $\mathcal{PC}(I)$ is uniformly simple.

\smallskip

\begin{defi} Let  $a\in I=[0,1)$ and $G$ be a subgroup of $\mathcal{PC}(I)$.
\begin{itemize}
\item We say that $G$ is $\boldsymbol{a-}$\textbf{\textit{LBS}} ($\boldsymbol{a-}$\textbf{L}ocally \textbf{B}oundedly \textbf{S}upported) if for every $g\in G$ and every $a$-proper interval $J$ such that $\widehat{g}$ is continuous on $J$ and $\widehat{g}(J)$ is $a$-proper, there exists $g_a \in BG_a$ such that $\widehat{g}\vert_{J}=\widehat{g_a}\vert_{J}$.
\item Let $J$ be a subinterval of $I$, we say that $G$ is $\boldsymbol{(a,J)}-$\textbf{proximal} if for every $a$-proper interval $K$ there exists $k\in G$ such that $\widehat{k}(K)\subset J$. 
\item We say that $G$ is $\boldsymbol{a-}$\textbf{proximal} if for every subinterval $J$ of $I$, the group $G$ is $(a,J)$-proximal.
\item We say that $G$ is \textbf{\textit{NCI}} (\textbf{N}on \textbf{C}ommuting \textbf{I}nvolution) if for any involution $i\in G$, there exists $h\in G$ such that  $i$ and $h i h^{-1}$ do not commute, it means that $i\circ (h i h^{-1})$ is not an involution.
\end{itemize}
\end{defi}

\begin{rema}\label{remNCI} \ 
\begin{itemize}
\item If $G$ is infinite and simple then $G$ is NCI. Indeed by absurd, the simplicity of $G$ implies that $G$ coincides with its normal subgroup generated by $\{h i h^{-1},h\in G \}$ which is abelian, but abelian simple groups are cyclic of prime order.
\item If $G$ is perfect and non trivial then $G$ contains elements that are not involutions. Indeed, if any element is an involution then every $g\in G$ can be written as a product of elements of the form $[i_l,j_l]= (i_l j_l)^2= \id$ and hence $g=\id$.
\end{itemize} 
\end{rema}

\smallskip
Our results on uniform simplicity are the following
\begin{theo} \label{th0} Let $a\in I$ and $G$ be a perfect and $a$-proximal subgroup of $B\mathcal{PC}(I)_a$.
\begin{itemize}
\item If $G$ does not contain involution then $G$ is $8$-uniformly simple.
\item If $G$ has the NCI property then $G$ is  $16$-uniformly simple.
\end{itemize}
\end{theo}

\begin{theo} \label{th1} Let $a\in I$ and $G$ be an $a$-LBS subgroup of $\mathcal{PC}(I)$ such that 
\begin{enumerate}
\item the regular $G$-orbit of $a$ is infinite and
\item the subgroup $BG_a$ is perfect and $a$-proximal.
\end{enumerate}
\begin{itemize} 
\item If $G$ does not contain involution then $G$ is $12$-uniformly simple.
\item If $G$ has the NCI property then $G$ is  $24$-uniformly simple.
\end{itemize}
\end{theo}

\smallskip

The hypotheses of Theorem \ref{th0} are closely related to the ones of Theorems 1.1 and 5.1 of Gal and Gismatullin in \cite{GG}. However, their theorems that concern either boundedly supported order preserving actions or full groups actions on a Cantor set do not apply directly to all subgroups of $\mathcal{PC}(I)$. Consequences of Theorem \ref{th1} are
 
\begin{coro}\label{coro1} The groups $\mathcal {PC}(I)$ and  $\mathcal {PC}^+(I)$ are uniformly simple.\end{coro}

\begin{coro}\label{coro2} The groups $PL^+(\mathbb S^1)$, $\mathcal A(I)$ and $\mathcal A^+(I)$ are uniformly simple.\end{coro}

\smallskip

Theorems \ref{th0} and \ref{th1} apply to certain Thompson like groups. They imply that the commutator subgroups of the Brown-Thompson groups $F_n$ and the Higman-Thompson groups $V_n$ are uniformly simple. This was proved in \cite{GG} with smaller bounds. Moreover, Theorem \ref{th1} applies to some Stein-Thompson groups $\TS$, in particular to the Thompson group $T_2$. The uniform simplicity of these groups can not be obtained by Gal and Gismatullin's results. 

\smallskip

\begin{coro} \label{coro3}
The Thompson group $T_2$, the Stein-Thompson groups $T_{\{n_1, n_2, ..., n_p\}}$ with $n_2= n_1^{2k}-1$ and in particular, $T_{\{2,3,\cdots \}}$, are uniformly simple.
\end{coro}

\begin{rema} \ 

{Theorem \ref{th1} does not apply to subgroups of $\homeo^+(I)$, since its Hypothesis (1) implies that $a\not=0$ and the $a$-proximality excludes the possibility that $0$ might be a global fix point.}

{In addition, a simple subgroup $G$ of $\homeo^+(I)$ that contains an $f$ having support in some $[c,d]$ with $0<c<d<1$ is a subgroup of $B\homeo^+(I)_0$ and might satisfy the hypotheses of Theorem \ref{th0}. {\footnotesize Indeed, Let $G$ be a simple group and $f\in G$ having support in some $[c,d]$ with $0<c<d<1$, then any $g\in G$ belongs to the normal closure of $\langle f \rangle$, that is $g$ is the product of $p$ conjugates $k_i f^{\pm 1} k_i^{-1}$ of $f$ or $f^{-1}$. Therefore $g$ has support in $[\min\limits_ik_i(c), \max\limits_i k_i(d)]$ so $g\in B\homeo^+(I)_0$.}}

{In conclusion,  Theorems \ref{th0} and \ref{th1} do not apply to {simple} subgroups of $\homeo^+(I)$ whose elements have dense supports.}
\end{rema}

\smallskip

Finally, going back to the O'Farell and Short questions mentioned above, if $G$ is one of the groups considered in Corollaries 
\ref{coro1}, \ref{coro2} and \ref{coro3}, then there exists a finite positive integer $n$ such that $G=\mathcal I^n=\mathcal R^n$.

\medskip

\noindent \textbf{Acknowledgements.} We express our gratitude to P. Arnoux for writing the appendix. We thank E. Ghys for communicating us, a long time ago, how  Proposition 1 of \cite{Va} allows to conclude that certain groups of homeomorphisms are uniformly perfect. His argument is reproduced in Section 2. We thank Y. Cornulier and O. Lacourte for fruitful discussions. We acknowledge support from the MathAmSud Project GDG 18-MATH-08, the Labex CEMPI (ANR-11-LABX-0007-01), the Universities de Lille and de la Rep\'ublica, the I.F.U.M. and the project ANR Gromeov (ANR-19-CE40-0007). The second author also thanks CNRS for the  d\'el\'egation during the academic year 2019/20.


\section{Uniform perfectness}

\subsection{Uniform perfectness for subgroups of $\boldsymbol{B\mathcal{PC}(I)_a}$.}  \ 

\begin{defi} \ \label{def2.1}

$\bullet$ Two subsets $S_1$ and $S_2$ of a group $G$ are \textbf{commuting} if any $\alpha\in S_1$ commutes with any $\alpha'\in S_2$. 

$\bullet$ Given $J$ a subset of $I$, we denote $\widehat{\mathcal{PC}}(J)=\{g\in \widehat{\mathcal{PC}}(I) : \supp(g)\subset J\}$ and we define $\mathcal{PC}(J)$ to be the image of $\widehat{\mathcal{PC}}(J)$ in $\mathcal{PC} (I)$ by the quotient morphism.
\end{defi}

\smallskip

In this section, inspired by the proof of Proposition 1 of Dennis and Vaserstein (\cite{Va}), we establish

\smallskip

\begin{prop}\label{UPBG}  Let $a\in I$ and $G$ be a subgroup of $B\mathcal{PC}(I)_a$. Suppose that there exist $f \in G$ and a subinterval $J \subset (0,1)$ such that $J$, $\widehat{f}(J)$ and $\widehat{f} ^{-1}(J)$ are pairwise disjoint and $G$ is $(a,J)$-proximal then any element of $[G,G]$ is the product of $2$ commutators in $G$. \end{prop}

\begin{proof} As Dennis and Vaserstein have noted, it suffices to prove that any product of $3$ commutators is the product of $2$ commutators. 

Let $g=\gamma_1 \gamma_2 \gamma_3$ with $\gamma_i=[a_i,b_i]$. By definition of $B\mathcal{PC}(I)_a$, there exists an $a$-proper interval $K$ which contains the support of all $\widehat{a_i},\widehat{b_i}$ and by $(a,J)$-proximality, there exists $k$ in $G$ such that $\widehat{k}$ sends $K$ into $J$. Thereby, conjugating by $k$, we can suppose that the supports of $\widehat{a_i}, \widehat{b_i}$ are included in $J$. For $\gamma, h \in G$, we denote by $\C_h(\gamma)= h \gamma h^{-1} $. 

Note that $\widehat{\mathcal{PC}}(J)$, $\widehat{\mathcal{PC}}(\widehat{f}(J))$ and $\widehat{\mathcal{PC}}(\widehat{f}^{-1}(J))$ are pairwise commuting subgroups of $\widehat{\mathcal{PC}} (I)$ since $J$, $\widehat{f}(J)$ and $\widehat{f}^{-1}(J)$ are pairwise disjoint. Thus, for any $w_0,w_1,w_2 \in \langle a_i,b_i\rangle$, it holds that $w_0$, $\C_f(w_1)$ and  $\C_{f^{-1}}( w_2)$ are pairwise commuting. Therefore
$$g=\gamma_1 \gamma_ 2 \gamma_ 3 =\gamma_1  \ \  \C_f(\gamma_2) \  \C_{f^{-1}}(\gamma_3) \ \  \C_{f^{-1}} (\gamma_3 ^{-1}) \ \C_f( \gamma_2^{-1})  \ \   \gamma_2 \gamma_3 = C_1 C_2 \text{, \  where}$$

\smallskip

$\bullet$ $C_1= \gamma_1 \  \C_f(\gamma_2) \  \C_{f^{-1}} (\gamma _3)$ is a commutator as a product of commutators of pairwise commuting pairs.
\medskip

$\bullet$  $C_2=  \C_{f^{-1}}(\gamma_3 ^{-1})  \  \C_f(\gamma_2^{-1})  \  \gamma_2 \gamma_3 =
\C_{f^{-1}}( \gamma_3 ^{-1})  \gamma_2   \ \  \C_f( \gamma_2^{-1})   \gamma_3$.

\medskip

Noticing that $\ds \C_f( \gamma_2^{-1})   \gamma_3 =  \C_f \left( \gamma_2^{-1}  \ \C_{f^{-1}}(\gamma_3)  \right ) =
 \C_f \left(   (\C_{f^{-1}}( \gamma_3 ^{-1})  \gamma_2 ) ^{-1}\right)$, we conclude that $C_2$ is a commutator as a product of an element by a conjugate of its inverse. \end{proof}

\subsection{Uniform perfectness for subgroups of $\boldsymbol{\mathcal{PC}(I)}$.} \

In this section, we prove a lemma that will make the link between the uniform perfectness of $G< \mathcal {PC} (I)$ and the one of its subgroups $BG_a$.

\begin{lemm}\label{dec} \ 

If $G< \mathcal {PC} (I)$ is $0$-LBS then for any $g \in G$ and $a\in I\setminus \disc(\widehat{g}) \cup \{0,\widehat{g}^{-1} (0) \}$, there exist $g_0 \in BG_0$ and $g_a \in BG_a$ such that $g=g_0 g_a$.
\end{lemm}

\begin{proof} 
Let $g \in G$. As $a\not=0$, $\widehat{g}(a)\not= 0$ and $a\notin \disc(\widehat{g})$, there exists $\delta>0$ such that $\widehat{g}$ is continuous on $V_{\delta}(a)$ and $V_{\delta}(a)$ and $\widehat{g}(V_{\delta}(a))$ are $0$-proper intervals.

Since $G$ is $0$-LBS, there exists $g_0 \in BG_0$ such that \ $\widehat{g}\vert_{V_{\delta}(a)}=\widehat{g_0}\vert_{V_{\delta}(a)}$, thereby \break  $(\widehat{g_0} ^{-1} \circ \widehat{g}) \vert_{V_{\delta}(a)}=\id\vert_{V_{\delta}(a)}$ and then $g_a:=g_0 ^{-1} \circ g \in BG_a$. \end{proof}

\section{The Burago and Ivanov method {\footnotesize{(adapted from Lemma 3.6 and 3.8 of \cite{BI})}}}

\begin{defi}  Let $G$ be a group and $f\in G$. An \textbf{$f$-commutator} is an element of the form $[\tilde f, h]$ for some $h\in G$ and some $\tilde f$ conjugate to $f$ or $f^{-1}$. \end{defi}

\begin{rema}\label{remfcom} 
Any conjugate of an $f$-commutator is an $f$-commutator. All elements of the form 
$[h,f]$ and $[h,f^{-1}]$ are $f$-commutators. Any $f$-commutator is product of $2$ conjugates of $f$ or $f^{-1}$. 
\end{rema}

\begin{prop}\label{BI} Let $G$ be a group of bijections of a space $X$. Let $f \in G$ and $\Omega \subset X$ such that $\Omega$, $f(\Omega)$ and $f^2(\Omega)$ are pairwise disjoint. Let $g_1$, $g_2$ and $k $ be elements of $G$ such that $k(\supp(g_1)\cup \supp(g_2) ) \subset \Omega$ then $[g_1,g_2]$ is a product of two $f$-commutators.
\end{prop}

\begin{proof} Let us recall that $\C_f(w)= fwf^{-1}$. We first prove the following
\begin{lemm} \label{fcom} If $\supp(g_i) \subset \Omega$ for $i=1,2$ then 
$$(\star) \ \ \ \  g_1 g_2 \ \C_f(g_1 ^{-1}) \ \C_{f^2}(g_2 ^{-1})=[  \  \C_{f}( g_2) \ g_1g_2 \ , \ f \ ]$$
is an $f$-commutator.
\end{lemm}

$$\text{Indeed, } \ \ \ \ \ \ \ \ [\ \C_{f}( g_2) \ g_1g_2 \ , \ f \ ] \ =  \ \C_f(g_2) \  g_1 g_2  \ f \ \ \ (g_1 g_2)^{-1} \ \C_f(g_2^{-1}) \  f^{-1} \ \ $$
$$= \ \C_f(g_2)  \ g_1 g_2 \ \C_f((g_1 g_2)^{-1}) \  C_{f^2}(g_2^{-1} ).$$ \hskip 1.5 truecm Since $g_1 g_2$ and $\C_f(g_2)$ have disjoint supports, they  commute and we get $$=\ g_1 g_2 \  \C_f(g_1^{-1}) \  \C_{f^{2}}(g_2^{-1}).$$

\begin{rema}\label{fcoms} Writing $(\star)$  for $g_1^{-1}$ and $g_2^{-1}$, we get that $\ds g_1^{-1} g_2^{-1} \ C_f(g_1) \ C_{f^2}(g_2)$ is also an $f$-commutator. \end{rema}

\medskip
We turn now on to the proof of Proposition \ref{BI}.
Let $g_1,g_2 \in G$  and $k\in G$ such that $k(\cup \supp(g_i) ) \subset \Omega$. The commutator  $[g_1,g_2]$ can be written as $\C_{k^{-1}} ([\C_{k}(g_1),\C_{k}(g_2)])$, with $\C_{k}(g_i)$ of support in $\Omega$. So by Remark \ref{remfcom}, we can suppose w.l.o.g that the $g_i$ have support in $\Omega$ and we have:
$$[g_1,g_2] \ = \ g_1 g_2 \ \ \ \C_{f}(g_1^{-1}) \ \C_{f^2}(g_2^{-1}) \ \ \ \C_{f^2}(g_2) \ \C_{f}(g_1) \ \ g_1^{-1}g_2^{-1}.$$
Since $g_1^{-1} g_2^{-1}$, $\C_f(g_1)$ and $\C_{f^2}(g_2)$ have pairwise disjoint supports, they all commute and we get
$$[g_1,g_2] \ = \ g_1 g_2 \ \C_{f}(g_1^{-1}) \ \C_{f^2}(g_2^{-1}) \ \ \ \ g_1^{-1} g_2^{-1} \ \C_{f}(g_1) \ \C_{f^2}(g_2)$$

Finally, by Lemma \ref{fcom} and Remark \ref{fcoms}, $[g_1,g_2]$ is a product of two $f$-commutators. \end{proof}


\section{Uniform simplicity, proof of Theorems \ref{th0} and \ref{th1}} 
We first show two lemmas which ensure that Propositions \ref{UPBG} and \ref{BI} will apply. 

\begin{lemm} \label{lem3.1} 

Let $f\in \mathcal{PC}(I)$ such that $f^2\not =\id$, then there exists $J =J_{\widehat{f}} \subset (0,1)$ such that $J$, $\widehat{f}(J)$ and ${\widehat{f}}\hskip 1truemm ^2(J)$ are pairwise disjoint subintervals.
\end{lemm}

Indeed, as $f^2\not=\id$ in $\mathcal{PC}(I)$, the support of $\widehat{f}\hskip 1truemm ^2$ is infinite and there exists $a\in \supp(\widehat{f}\hskip 1truemm ^2)$ a continuity point of both $\widehat{f}$ and $\widehat{f}\hskip 1truemm ^2$. The required statement follows from a standard argument of continuity.

\begin{lemm} \label{lem3.2} Let $a\in I$ and $G$ be an $a$-proximal subgroup of $B\mathcal{PC}(I)_a$. Then for all $g,h \in G$ and any subinterval $J$ there exists $k\in G$ such that $\widehat{k}(\supp(\widehat{g}) \cup \supp(\widehat{h})) \subset J$.
\end{lemm}

\begin{proof} Arguments are analogous to ones of the proof of Proposition \ref{UPBG}. However for completeness we reproduce them. As $\widehat{g},\widehat{h} \in B\widehat{\mathcal{PC}}(I)_a$, there exists an $a$-proper interval $K$ that contains both supports of $\widehat{g}$ and $\widehat{h}$. By the $a$-proximality of $G$, there exists $k\in G$ such that $\widehat{k}$ sends  $K$ into $J$.  
\end{proof}

\subsection{Proof of Theorem \ref{th0}.} Let $G <B\mathcal{PC}(I)_a$ and $f, \phi\in G\setminus\{\id\}$. 

If $f$ is not an involution, by Lemma \ref{lem3.1}, there exists $J\subset (0,1)$ such that $J$, $\widehat{f}(J)$ and $\widehat{f} \hskip 1truemm ^2(J)$ are pairwise disjoint intervals.

Since $G$ is perfect, $\phi\in [G,G]$. Thus, $G$ being $a$-proximal, Proposition \ref{UPBG} (changing $J$ for $\widehat{f}(J)$) implies that  $\phi$ is a product of $2$ commutators $[g_i,h_i]$, $i=1,2$.

In addition, by Lemma \ref{lem3.2}, the interval $J$ and the maps $\widehat{g_i}$ and $\widehat{h_i}$ satisfy the hypotheses of Proposition \ref{BI}, hence their commutator is a product of two $\widehat{f}$-commutators and applying the quotient morphism, each $[g_i,h_i]$ is a product of two $f$-commutators. Then $\phi$ is a product of $4$ $f$-commutators. As any $f$-commutator is a product of $2$ conjugates of $f$ or $f^{-1}$, we finally get that $\phi$ is a product of $8$ conjugates of $f$ or $f^{-1}$.

\medskip

If $f$ is an involution, NCI-property implies that there exists $h\in G$ such that $F=f\circ (h f h^{-1})$ is not an involution. Applying the previous case to $F$ we get that $\phi$ is the product of $8$ conjugates of $F$ or $F^{-1}$ that is a product of $16$ conjugates of $f=f^{-1}$. 

\medskip

\subsection{Proof of Theorem \ref{th1}.} \label{Pth1} 

W.l.o.g, we can suppose that $a=0$ and  we begin by proving 
\begin{lemm} \label{regorb} 
If $BG_0$ is perfect and $0$-proximal then for every $a=\widehat{g}(0)\in G_{reg}(0)$ it holds that $BG_a = g BG_0 g^{-1}$ is perfect and $a$-proximal.
\end{lemm} 

Indeed, we first prove that $BG_a = g BG_0 g^{-1}$, this will immediately imply that $BG_a$ is perfect. Let $f\in BG_a$, by definition $\supp(\widehat{f}) \subset I\setminus V_{\delta} (a)$ for some positive $\delta$. By the continuity of $\widehat{g}$ at $0$ there exists $\eta>0$ such that $\widehat{g}(V_{\eta} (0)) \subset V_{\delta} (\widehat{g}(0))= V_{\delta} (a)$, thereby $I\setminus V_{\delta} (a) \subset \widehat{g}(I\setminus V_{\eta} (0))$. Therefore $\supp(\widehat{g}^{-1} \circ \widehat{f} \circ \widehat{g}) =\widehat{g}^{-1}(\supp(\widehat{f})) \subset I\setminus V_{\eta} (0)$, then $g^{-1} \circ f \circ g \in BG_0$ and finally $f\in g BG_0 g^{-1}$. The other inclusion is obtained by interchanging the roles of $a$ and $0$.

Suppose now that $BG_0$ is $0$-proximal. Let $J$ be a subinterval of $I$, $J'$ be a subinterval of $\widehat{g}^{-1}(J)$ and $K_a$ be an $a$-proper interval. Therefore, as previously, $\widehat{g}^{-1}(K_a) \subset I\setminus V_{\eta} (0)$ for some positive $\eta$. We conclude from the $0$-proximality of $G$ that there exists $k_0$ such that $\widehat{k}_0( I\setminus V_{\eta} (0)) \subset J'$, hence that $\widehat{k}_0 \circ \widehat{g}^{-1}(K_a) \subset  J'$ and finally that  $\widehat{g} \circ \widehat{k}_0 \circ \widehat{g}^{-1}(K_a) \subset \widehat{g}(J') \subset J$ with  $g \circ k_0 \circ g^{-1}\in BG_a$.

\medskip

We turn now on to the proof of Theorem \ref{th1}. Let $G <\mathcal{PC}(I)$ and $f,\phi\in G\setminus\{\id\}$.

By Hypothesis (1), the regular $G$-orbit of $0$ is infinite, therefore it contains some point $a \notin\{0, \widehat{\phi}^{-1}(0)\}\cup \disc (\widehat{\phi})$. Since $G$ is $0$-LBS, Lemma \ref{dec} implies that there exist $\phi_0 \in BG_0$ and $\phi_a \in BG_a$ such that $\phi=\phi_0 \phi_a$.

We claim that $\phi= g_0 b_a$ with $g_0 \in BG_0$ and $b_a$ a commutator in $BG_a$.
\begin{quote} {\footnotesize{ \ 

\vskip -3truemm
Indeed, by the definition of $BG_a$, there exist $K_a$ an $a$-proper interval and $\delta>0$ such that $\supp(\widehat{\phi_a}) \subset K_a \subset I\setminus V_{\delta}(a)$.

According to Remark \ref{regorb}, the group $BG_a$ is $a$-proximal. Then given any $\eta>0$, there exists $k_a \in BG_a$ such that $\widehat{k_a} (K_a)\subset I \setminus V_\eta(0)$. Therefore $\supp(C_{\widehat{k_a}}(\widehat{\phi_a})) = \widehat{k_a}(\supp(\widehat{\phi_a})) \subset  I\setminus V_{\eta}(0)$ and then $C_{k_a}(\phi_a) \in BG_0$. Finally $$\ds \phi=\phi_0 \phi_a= \phi_0 C_{k_a}(\phi_a) C_{k_a}(\phi_a^{-1})\phi_a =g_0 b_a\text{ with}
\left\{ \begin{array}{ll}
g_0=\phi_0 C_{k_a}(\phi_a) \in BG_0 \text{ and} \cr
b_a=C_{k_a}(\phi_a^{-1})\phi_a \text{ a commutator in } BG_a.
\end{array}\right.$$
}} \end{quote}

\smallskip

As $BG_0$ is perfect and $0$-proximal, it is non trivial. Hence, Remark 3 ensures that it contains some $f'$ that is not an involution.

In addition, as $BG_0$ is $0$-proximal then Lemma \ref{lem3.1} and Proposition \ref{UPBG} (changing $J$ for $\widehat{f'}(J)$) imply that $g_0$ is a product of $2$ commutators in $BG_0$. Therefore $\phi$ is a product of $2$ commutators in $BG_0$ and one commutator in $BG_a$ for some $a\in I$.

If $f$ is not an involution, applying Lemma \ref{lem3.2} to $BG_0$ and $BG_a$ (which is $a$-proximal by Remark \ref{regorb}) and Proposition \ref{BI} to $G$, we obtain that each such commutator is a product of two $f$-commutators. As every $f$-commutator is a product of $2$ conjugates of $f$ or $f^{-1}$, we finally get that $\phi$ is a product of $12$ conjugates of $f$ or $f^{-1}$. 

\medskip

If $f$ is an involution, by NCI-property, there exists $h\in G$ such that $F=f\circ (h f h^{-1})$ is not an involution. Applying the previous case to $F$, we get that $\phi$ is the product of $12$ conjugates of $F$ or $F^{-1}$ that is a product of $24$ conjugates of $f$. \hfill $\square$

\section{Proof of Corollaries.}

In this section, we check that many groups satisfy the hypothesis of Theorem \ref{th1}.

According to Arnoux and Lacourte $\mathcal{PC}^+(I)$, $\mathcal{PC}(I)$ and $\mathcal A(I)$ are simple (see \cite{Ar1} III Proposition 1.7 and \cite{Lac} Theorem 1.4). In Section \ref{appendix1} we will prove that $\mathcal A^+(I)$ is simple. 

Epstein (\cite{Ep}) established that $PL^+(\mathbb S^1)$ is simple and Thompson showed that $T_2$ is simple (see e.g. \cite{CFP}) and in Section \ref{appendix2} we will prove that the Stein-Thompson groups $T_{(n_1, n_2, ..., n_p)}$ with $n_2= n_1^{2k}-1$ are simple.

\smallskip 

All the groups previously mentioned are infinite, hence they are NCI, by Remark \ref{remNCI}. It is easy to check that they  also are $0$-LBS, the regular $G$-orbit of $0$ is infinite and have an associated $BG_0$ which is $0$-proximal.

\smallskip 

It remains  to prove that the corresponding $BG_0$ are perfect. 

\noindent If $G=PL^+(\mathbb S^1)$ then $BPL^+(\mathbb S^1)_0=[PL^+(I),PL^+(I)]$ is perfect, by Epstein (\cite{Ep}).

\noindent If $G=T_2$ then by Theorem 4.1 of \cite{CFP} and related comments, $(BT_2)_0=[F_2,F_2]$ is perfect. If $G=T_{(n_1, n_2, ..., n_p)}$ with $n_2= n_1^{2k}-1$, this is provided by Lemma \ref{lemma :3}.

Finally, let $G\in \{ \mathcal{PC}(I), \mathcal{PC}^+(I), \mathcal A(I), \mathcal A^+(I)\}$ and $f_0$ in $BG_0$. There exist $c,d$ with $0<c<d<1$ such that $\supp(f_0) \subset [c,d)$. The group $G([c,d))$ of elements of $G$ with support in $[c,d)$ is isomorphic to $G$  which is a simple group. In particular, $G([c,d))$ is perfect and we get that $f_0$ is a product of 
commutators of elements in $G$ having support in $[c,d)$, hence of elements in $BG_0$.

\section{ $\mathcal A^+(I)$ is simple}\label{appendix1}

\subsection{Preliminaries.} 

The aim of this section is to fix notation and terminology, to collect a few results and to prove some basic results to be used for establishing the simplicity of $\mathcal A^+(I)$. In particular, we describe the conjugacy classes of involutions in $\mathcal A^+(I)$.

\medskip

Definition \ref{defAIET} can be extended to every half open real interval $J$ (see the appendix by P. Arnoux) and the corresponding groups are denoted by $ PL^+(J) <PL^+(\mathbb S_J) <\mathcal A^+(J)$, where $\mathbb S_J$ is the circle obtained by identifying the endpoints of $J$ and $PL^+(J)$ is identified with the stabilizer of the left endpoint of $J$ in  $PL^+(\mathbb S_J)$. It is plain that $PL^+(\mathbb S_J)$ is isomorphic to $PL^+(\mathbb S^1)$.

\begin{defi} \

$\bullet$ An IET that has at most one interior discontinuity point is called a {\bf rotation} and it is denoted by $R_a$, where $a$ is the image of $0$.

$\bullet$ An IET $g$ whose support is a half-open interval $J=[a,b)\subset[0,1)$ is a {\bf restricted rotation} if the direct affine map that sends $J$ to $[0,1)$ conjugates $g_{\vert J}$ to a rotation. We denote it by $R_{\alpha,J}$, where $\alpha$ is defined by $R_{\alpha,J} (x) =x+\alpha  \ (mod \ \vert b-a \vert)$ for $x\in J$. 
\end{defi}

\begin{lemm}\label{clcoin}  \

Every non trivial involution $i\in \mathcal A^+(I)$ is conjugated in $\mathcal A^+(I)$ to either $R_{\frac 1 2}$ or to $RR_{\frac 1 2}$ the order $2$ restricted rotation of support $[{\frac 1 2},1)$ that exchanges $[{\frac 1 2},{\frac 3 4})$ and $[{\frac 3 4}, 1)$.
\end{lemm}
\begin{proof}
As $i$ is a non trivial involution, the interval $I$ can be decomposed into a finite union of pairwise disjoint half-open intervals: $I_1, \cdots , I_p$ and $J_1, \cdots , J_q$ satisfying the following: 
\begin{enumerate}
\item The map $i$ is continuous on these intervals. 
\item The integers $p$ and $q$ are such that \ $p=2k \geq 2$, \ $q\geq 0$ and in the case that $q=0$ there is no $J_j$.
\item $J_j \subset Fix(i)$ and $i(I_{j})= I_{j+k}$.
\end{enumerate}
 
\medskip
 
Let $H$  be the AIET defined by:

\begin{itemize}
\item Whenever $q\not=0$, the map $H$ sends affinely $J_j$ to $[{\frac{j-1}{2q}}, {\frac{j}{2q}})$ for $j=1, \cdots , q$.

\item $H$ sends affinely $I_j$ to \ \ $\left\{ \begin{array}{cc}
[ \frac{j-1}{p}, \frac{j}{p}) &{\text { for } } j=1, \cdots , p {\ \text { if } } q =0,   \cr
[{\frac 1 2}+\frac{j-1}{2p}, {\frac 1 2}+\frac{j}{2p})  &{\text { for } } j=1, \cdots , p  {\ \text { if }  } q\not=0. \end{array}\right.$
\end{itemize}
 
\smallskip
 
We can check that $H$ conjugates $i$ to a map with support $[0,1)$ if  $q =0$ or $[{\frac 1 2},1)$ if $q\not=0$ which also is an IET (this can be verified by computing the slope of $H\circ i \circ H^{-1}$ on each $H(I_j)$). Moreover by definition, $H\circ i \circ H^{-1}$ sends any two cyclic-consecutive intervals among the $H(I_{j}),\  j=1, \cdots , k$ to cyclic-consecutive ones so it is continuous except at ${\frac 1 2}$ if  $q =0$ and  at ${\frac 1 2}$ and ${\frac 3 4}$ if $q \not=0$. 

\smallskip

In conclusion, $H\circ i \circ H^{-1}=R_{\frac 1 2}$ if $q =0$ or $H=RR_{\frac 1 2}$ if $q \not=0$.  \end{proof}


\subsection{The group $\boldsymbol{\mathcal A^+(I)}$ is perfect and generated by its involutions}  \ 

We first exhibit generators of $\mathcal A^+(I)$.

\begin{prop} \label{PLAIET} \ 

Every $f \in \mathcal A^+(I)$ can be written as $f= g\circ h$ with $h\in PL^+(I)$ and $g$ an IET.
\end{prop}

\begin{proof}
Let $f \in \mathcal A^+(I)$, we denote by $I_1, \cdots, I_p$ the maximal continuity intervals of $f$ and we denote by  $J_{\pi(i)}$ the interval $f(I_i)$. We consider the IET $E$ defined by the partition $\{J_i\}$ and the permutation $\pi^{-1}$ that tells us how the $J_i$ are rearranged. By construction, the AIET $h=E\circ f$ is continuous on $I$ and $f= E^{-1} \circ h$ has the required form.
\end{proof}

According to \cite{Ar1}, \cite{No} or \cite{Vo} (see the appendix for a proof), any interval exchange transformation $g$ is a product of restricted rotations. Therefore, Proposition \ref{PLAIET} insures that every $f \in \mathcal A^+(I)$ is a product of commutators [resp. involutions] if this property holds for any $h \in PL ^+(\mathbb S^1)$.

Indeed, $PL ^+(I)$ is a subgroup of $PL^+(\mathbb S^1)$. In addition, the map $f\mapsto f \vert_J$ sends the restricted rotations of support $J$ into $PL ^+(\mathbb S_J)$ and it is an isomorphism onto its image, the subgroup of $PL ^+(\mathbb S_J)$ consisting of its rotations. Therefore, writing a restricted rotation of support $J$ as product of commutators [resp. involutions] reduces to do that for a rotation in the group $PL ^+(\mathbb S_J)$ which is isomorphic to $PL^+(\mathbb S^1)$.

\smallskip

As Theorem 3.2 of \cite{Ep} states that $PL^+(\mathbb S^1)$ is simple, $PL^+(\mathbb S^1)$ is generated by either its commutators or its involutions, so $$\mathcal A^+(I)= \langle \text{\footnotesize{commutators}} \rangle=\langle \text{\footnotesize{involutions}}\rangle.$$

\subsection{The group $\boldsymbol{\mathcal A^+(I)}$ is simple} \ 

\smallskip

Let $N$ be a non trivial normal subgroup of $\mathcal A^+(I)$. The problem reduces to prove that $N$ contains a fix point free involution $\tau_1$ and a non trivial involution $\tau_2$ having fix points since $\mathcal A^+(I)=\langle \text{\footnotesize{involutions}} \rangle$ will be the normal closure of $\langle \tau_1, \tau_2 \rangle$, by Lemma \ref{clcoin}.

\medskip

Let $f$ be a non trivial element of $N$, then there exists a non empty half-open interval $J$ such that $f(J)\cap J =\emptyset$ and $J$ and $f(J)$ have length less than ${\frac 1 2}$.

\smallskip

Let $i\in \mathcal A^+(I)$ be an involution with support $supp(i)=J$. Therefore $\supp(f \circ i \circ f^{-1})=f(\supp(i)) =f(J)$ is disjoint from $\supp(i)$. Consequently $f \circ i \circ f^{-1}$ and $i$ commute, hence $\tau_1 = [f,i]=f \circ i \circ f^{-1} \circ i $ is an involution of support $J\cup f(J)$ and it belongs to $N$. Then we have proved that $N$ contains a non trivial involution $\tau_1$ having fixed points.

\medskip

For constructing a fix point free involution in $N$, we consider $ h_1, h_2$ in $\mathcal A^+(I)$ such that $$(\star) 
\ \ \ h_1(J)=[0,{\frac 1 4}), \ \  h_1(f(J))=[{\frac 1 2 },{\frac 3 4}), \ \ h_2(J)=[ {\frac 1 4},{\frac 1 2}) \text{ and } \ \  h_2(f(J))=[{\frac 3 4 }, 1).$$

Note that, for $j=1,2$, the map $f_j=h_j \circ f \circ h_j^{-1}$ belongs to $N$ and satisfies $f_j(J_j) \cap J_j = \emptyset$,  where $J_j=h_j(J)$.

\smallskip

As in the previous case, we consider involutions $i_j$ for $j=1,2$ such that $\supp(i_j)=J_j$ and we get that $[f_j,i_j]$ is an involution of support $J_j \cup f_j(J_j)$ and it belongs to $N$.

\smallskip

From $(\star)$, we deduce $\supp([f_1, i_1]) = [0, {\frac 1 4}) \cup [{\frac 1 2 }, {\frac 3 4})$ and $\supp( [f_2, i_2]) =[ {\frac 1 4},  {\frac 1 2}) \cup [{\frac 3 4 }, 1)$.  Finally,  we obtain that $\tau_2=[f_1, i_1] [f_2, i_2] \in N $ is an involution of full support. Then we have also proved that $N$ contains a fix point free involution $ \tau_2$.

\section{Simplicity of certain Stein-Thompson groups}\label{appendix2}

In this section, we prove Theorem \ref{ThStTh}, using results of Stein (\cite{Ste}) and Bieri-Strebel's Lemma C12.8 and Theorem C12.14 of \cite{BiSt} that, in our context, can be stated as
\begin{theor}\label{BSBperf} The group $(B\TS)_0=\{f\in \TS : f(0)=0, f(1)=1,$ $Df(0)=Df(1)=1\}$ is perfect provided that the following properties $(i)$ and $(ii)$ hold.
\begin{enumerate}
\item[$\boldsymbol{(i)}$] \ $(1-\Lambda)A =A$, where $(1-\Lambda)A = \left\{ \  \sum \ (1- \lambda_i) a_i  \ ,  \ \lambda_i\in \Lambda= \langle n_i \rangle, \ a_i\in A=\mathbb Z[\Lambda] \ \right\}$. 
\item[$\boldsymbol{(ii)}$] \ $\Lambda$ contains a rational number $\frac{n}{q}> 1$ so that $n^2-q^2\in \langle n_i \rangle$.
\end{enumerate}
\end{theor}

\begin{lemm} \label{lemma :1}
Let $\Lambda= \langle n_i \rangle$, $A=\mathbb Z[\Lambda]$ and $d =\gcd (n_i-1)$. Then  $\ds (1-\Lambda)A= dA.$
\end{lemm}

\begin{proof} \ 

\noindent First, we prove the inclusion $(1-\Lambda)A \subset dA$.
 
By the definition of $(1-\Lambda)A $, it suffices to show that $(1-\lambda) a \in dA$, for any $a\in A$ and $\lambda = n_1^{s_1} ...n_p^{s_p} \in \Lambda$. By converting the fractions to have the same denominator,  there exist $q_i, t_i\in \mathbb N$ and $a'\in A$ such that $$(1-\lambda)a = (1- n_1^{s_1} ...n_p ^{s_p})a = (n_1^{q_1} ...n_p ^{q_p}- n_1^{t_1} ...n_p ^{t_p}) a'.$$  

By replacing the $n_i$'s by $k_i d +1$ and developing, we obtain $(1-\lambda) a = (dN) a' $ with $N\in \mathbb{N}$.

\smallskip

\noindent Next, we show that $dA\subset (1-\Lambda)A$.  

From Bezout's identity, we obtain $d = u_1(n_1-1) + \dots + u_p(n_p-1)$ with $u_i \in \mathbb{Z}$. 

Thus, for any $a\in A$, we have $da = \sum (n_i-1) (u_i a) \in (1-\Lambda)A$. \end{proof}


\begin{lemm}  \label{lemma :3} 
$(B\TS)_0$ is perfect provided that $\{ n_i \} = \{ n_1, n_1^{2k}-1,n_3, \cdots , n_p \}$.
\end{lemm}

\begin{proof} We check that $(B\TS)_0$ satisfies the properties $(i)$ and $(ii)$ of Bieri and Strebel's Theorem. Indeed, by Lemma \ref{lemma :1}, the property $(i)$ is equivalent to that $d=1$ which follows from the fact that $gcd(n_1 -1, n_1^{2k}-2)=1$. 

Moreover, considering $\ds \frac{n}{q}=\frac{n_1^k}{1}$ yields $n^2-q^2= n_1^{2k}-1 \in \Lambda$, so $(ii)$ holds for $(B\TS)_0$. \end{proof}

\begin{lemm} \label{lemma :2}
If $(B\TS)_0$ is perfect then $\TS$ is perfect.
\end{lemm}

\begin{proof} Let $f \in \TS$. As $\TS$ is a  $0$-LBS group, Lemma \ref{dec} with $a \in A\setminus \{0, f^{-1}(0)\}$ implies that $f= f_0 f_a$ with $f_0 \in (B\TS)_0$ and $f_a \in (B\TS)_a$.

It is a simple matter to prove that $(B\TS)_a= R_a (B\TS)_0 R_a^{-1}$ (see the proof of Lemma \ref{regorb}) and Lemma \ref{lemma :3} now implies that both $(B\TS)_0$ and $(B\TS)_a$ are perfect. We conclude that $f= f_0 f_a$ is a product of commutators in $\TS$ and finally that $\TS$ is perfect. 
\end{proof}

\medskip

We turn now on to the proof of the simplicity of $\TS$. According to \cite{Ste}, the group $\DDT$ is simple and $\DDT= \TS$ by the previous lemma, so we have that $\TS$ is simple.


\begin{appendix} 
\section{Simplicity of groups of interval exchange tranformations}
\begin{center} by Pierre Arnoux \end{center}

\bigskip

In this appendix, we prove the simplicity of some groups of interval exchange transformations; these results were obtained in \cite{Ar1}, but are not easily available. Recall the definitions:

\begin{definition} An interval exchange transformation on an interval $J=[a,b)$ is a bijection of $J$ which is everywhere right continuous, and, except on a finite number of points, continuous and derivable with derivative $1$\textcolor{red}{;} alternatively, it can be defined as a permutation by translations on a finite collection of semi-open subintervals of $J$.
 
More generally, an affine (resp. generalised) interval exchange transformation is a bijection defined by a finite partition of half open intervals, such that the restriction of the map to each interval is an orientation preserving affine map (resp. an orientation preserving homeomorphism).
 
An interval exchange transformation with flips is a bijection on $J$, except maybe for a finite set, which is derivable except for this finite set, with derivative $+1$ or $-1$. As noted in the introduction, it is defined up to a finite set.
\end{definition}

From now on, we fix an interval $J$. As before, we denote by $\iet$ the group of interval exchange transformations on the interval $J$, by $\ieta$ [resp. $\ietg$]  the group of affine [resp. generalised] interval exchanges transformations and  by $\ietf$ the group of classes of interval exchange transformations with flips.
 
In this appendix, we prove the following : 

\begin{proposition}\label{simpgrp}
The groups $\ieta, \ietg$ and  $\ietf$ are simple. The group $\iet$ is not simple, but its commutator subgroup is simple.
\end{proposition}

The proof of the proposition consists, using a lemma due to Epstein, in proving first that the commutator subgroup of all these groups is the smallest normal subgroup, and then, for the first three, in proving that they are perfect.

\subsection{A condition implying that every normal subgroup contains the commutator subgroup}

Recall that two transformations with disjoint support commute.

Remark that, if $H$ is a normal subgroup of a group $G$, and $h\in H$, then for all $a\in G$, $[a,h]=aha^{-1} h^{-1} $ is in $H$, as product of two elements of $H$: $aha^{-1}$ which is a conjugate of $h$, hence in $H$ by normality, and the inverse of $h$. Remark also that, if $a$ commutes with $c$, then $[a,bc]=abca^{-1}c^{-1}b^{-1}= [a,b]$.
We will use these properties to prove the following lemma, due to Epstein \cite{E68}

\begin{lemma} \label{lemmepst}
Let $G$ be a group of transformations of a manifold endowed with a measure $\mu$. Suppose that $G$ satisfies the two conditions: 
\begin{enumerate}
\item For all $\epsilon>0$, any element of $G$ is the product of a finite number of elements whose support has measure less than $\epsilon$.
\item For all $h\in G\setminus \{\id\}$, there exist $E\subset \supp(h)$ such that $h(E)\cap E=\emptyset$, and $\epsilon >0$ such that, if $g_1$ and $g_2$ are two elements of $G$ whose support has measure less than $\epsilon$, we can find $f\in G$ such that $f(\supp(g_i))\subset E$ for $i=1,2$.

Then $[G,G]$ is the smallest normal subgroup of $G$.
\end{enumerate}
\end{lemma}
\begin{proof}Let $H$ be a normal subgroup of $G$; we want to prove that any commutator belongs to $H$. Let $h$ be a non trivial element of $H$, and let $E$ and $\epsilon$ be as in condition (2). By condition (1), it is enough to prove that the commutator of two elements $g_1, g_2$ with support of measure less than $\epsilon$ belongs to $H$.

Let $f$ be as in condition 2, and $h'=f^{-1}hf$. We have $h'\in H$ by normality. If $S_i$, for $i=1,2$, is the support of $g_i$, one checks that $h'(S_i)\subset f^{-1}(h(E))$ is disjoint from $S_1\cup S_2\subset f^{-1}(E)$. This implies that $g_1$ and the conjugate $h'g_2^{-1}h'^{-1}$ of the inverse of $g_2$ have disjoint support, hence commute. This fact, and the remarks above, imply that $$[g_1,g_2]=[g_1, g_2h'g_2^{-1}h'^{-1}]= [g_1,[g_2,h']]\in  H$$

We have proved that the commutator of any element with small support belongs to $H$; since these elements generate $G$, the group of commutators is included in $H$.
\end{proof}

\subsection{Interval exchange transformations are product of transformations with small support}

We will prove that, in all the considered groups of interval exchange transformations, any element can be written as a product of a finite number of elements with a support of arbitrarily small size.

Now, we list some definitions and properties  that are easily available in \cite{GL21} and we add proofs for sake of completeness.

\begin{definition} Let $\alpha, \beta \in \overline{J}=[a,b]$ and  $0\leq \theta < \beta-\alpha$.

The \textbf{symmetry} of $[\alpha,\beta)$, denoted by $\mathcal I_{[\alpha,\beta)}$,  is the element of $\mathcal G(J)$ represented by the FIET $i=\widehat{\mathcal I_{[\alpha,\beta)}}$ given by  $i(x)=x \ \text{ if } \ \ x\notin (\alpha,\beta) \ \ \ \ \text{  and } \ \ \  i(x)=\alpha+\beta-x \ \text{ if } \  \ x\in (\alpha,\beta).$

A \textbf{distinguished involution} is a product of finitely many symmetries having disjoint supports.
\end{definition}

\begin{remark}\label{Stheta}
Let $\theta \in [0,1)$, set $R_\theta=R_{\theta, [0,1)}$ and $S_\theta=\mathcal I_{[0,\theta)} \circ \mathcal I_{(\theta,1)}$, it is easy to check that $S_\theta \circ S_{\theta'}=R_{\theta-\theta'}$ and $R_{\alpha} \circ S_{\theta} \circ R_{\alpha}^{-1} = S_{\theta + 2\alpha}$.
\end{remark}
\medskip

{\unitlength=0,35 mm
\begin{picture}(95,95)
\put(0,10){\textcolor{gray}{\line(1,0){95}}}
\put(10,0){\textcolor{gray}{\line(0,1){95}}}
\put (12, 85){$\widehat{\mathcal I_{[\alpha,\beta)}}$}
\put(0,0){$\scriptstyle a$} \put(95, 8){$\scriptstyle \vert$} \put(95,0){$\scriptstyle b$}
\put(44, 8){$\scriptstyle \vert$}
\put(38,0){$\scriptstyle  \alpha$}
\put(69, 8){$\scriptstyle \vert$}
\put(72,0){$\scriptstyle  \beta$}
\put (9,9){\textcolor{red}{$\scriptstyle  \bullet$}}
\put (42,42){\textcolor{red}{$\scriptstyle  \bullet$}}
\put (69,69){\textcolor{red}{$\scriptstyle  \bullet$}}
\put(10,10){\line(1,1){35}}
\put(45,70){\line(1,-1){25}}
\put(70,70){\line(1,1){25}}
\end{picture}}
{\unitlength=0,35 mm
\hskip 1.5 truecm  \begin{picture}(95,95)
\put(0,10){\textcolor{gray}{\line(1,0){95}}}
\put(10,0){\textcolor{gray}{\line(0,1){95}}}
\put (12, 85){$\widehat{\mathcal I_{[\alpha,b)}}$}
\put(0,0){$\scriptstyle a$}
\put(44, 8){$\scriptstyle \vert$}
\put(38,0){$\scriptstyle  \alpha$}
\put(95, 8){$\scriptstyle \vert$}
\put(95,0){$\scriptstyle  b$}
\put (9,9){\textcolor{red}{$\scriptstyle  \bullet$}}
\put (42,42){\textcolor{red}{$\scriptstyle  \bullet$}}
\put(10,10){\line(1,1){35}}
\put(45,95){\line(1,-1){50}}
\end{picture}}
{\unitlength=0,35 mm
\hskip 2 truecm  
\begin{picture}(95,95)
\put(0,10){\textcolor{gray}{\line(1,0){95}}}
\put(10,0){\textcolor{gray}{\line(0,1){95}}}
\put(12,85){$\widehat{S_{\theta}}$}

\put(0,0){$\scriptstyle 0$}
\put(54, 8){$\scriptstyle \vert$}
\put(52,0){$\scriptstyle \theta$}
\put(93, 8){$\scriptstyle \vert$}
\put(95,0){$\scriptstyle 1$}
\put(9,9){\textcolor{red}{$\scriptstyle  \bullet$}}
\put (53,53){\textcolor{red}{$\scriptstyle  \bullet$}}
\put(10,55){\line(1,-1){46}}
\put(95,55){\line(-1,1){39}}
\end{picture}}

\begin{lemma}\label{Lem6} Every element of $\iet$ is the product of a finite number of restricted rotations. Every element of $\ietf$ is the product of a distinguished involution and an element of $\iet$.
\end{lemma} 

\begin{proof} 

For clarity, given $K=[c,d)$ and $L=[d,e)$ two consecutive half-open intervals, we denote by $R_{K,L}$ the restricted rotation of support $K\sqcup L$ whose interior discontinuity point is $d$. Let $g\in \mathcal G^+(J)$ with continuity intervals $I_1,\cdots I_m$ and let $g(I_i)=J_{\pi(i)}$. We consider $R_1= R_{K,L}$, where $K=J_1 \cup \cdots \cup J_{\pi(1)-1}$ and $L=J_{\pi(1)}$. One directly has that $R_1\circ g \vert _{I_1}=\id$ and $g_1 := R_1\circ g \vert _{I_2 \cup \cdots \cup I_m}$ has at most $m-1$ continuity intervals.

Starting with $g_1$, we define similarly $R_2$ and we get that $R_2\circ g_1 \vert _{I_2}=\id$ and $g_2 := R_2\circ g_1 \vert _{I_3 \cup \cdots \cup I_m}$ has at most $m-2$ continuity intervals.

\smallskip

Repeating the previous argument $m-1$ times leads to a $g_{m-1}$ having at most $1$ continuity interval, so $g_{m-1}=\id$.

\smallskip

Extending the restricted rotations $R_i$ to $J$ by the identity map, we conclude that 

$\displaystyle R_{m-1} \circ \cdots \circ R_1 \circ g = \id$ and then $g$ is a product of a finitely many restricted rotations.

\medskip

Let $f\in \mathcal G (J)$, we denote by $I_1,\cdots I_m$  the continuity intervals of $f$ and by $$F=\left\{i\in \{1,...m\} :  f \mbox{ is orientation reversing on } I_i\right\}.$$ It is easy to check that $f \circ \prod_{i\in F} \mathcal I_{I_{i}} $ belongs to $\mathcal G^+(J)$ and that the $\mathcal I_{I_{i}}$'s have disjoint supports, so $\prod_{i\in F} \mathcal I_{I_{i}}$ is a distinguished involution and the second item of Lemma \ref{Lem6} directly follows.
\end{proof}

\begin{lemma} (Proposition \ref{PLAIET}) Every element of $\ieta$ (resp. $\ietg$) is the product of an element of $\iet$ and an orientation preserving PL homeomorphism of $J$ (resp. homeomorphism of $J$)
\end{lemma}

Proposition \ref{PLAIET} is only stated for $\ieta$, but the proof is exactly the same for $\ietg$.

\begin{lemma} For any $\epsilon >0$, any restricted rotation can be written in $\iet$ as the product of elements of $\iet $ with support of measure less than $\epsilon$.
\end{lemma}

\begin{proof}It suffices to prove it for a rotation on $[0,1)$. Let  $R_{\alpha}$  be a rotation on $[0,1)$. We can construct an element $f\in \iet$ with support on $[0,\frac 14)\cup R_{\alpha}[0,\frac 14]$, which coincides with $R_{\alpha}$ on $[0,\frac 14)$. The measure of $\supp(f)$ is less than half the measure of $\supp(R_{\alpha})$. Let $g=f^{-1}R_{\alpha}$; it is by construction the identity on $[0,\frac 14)$, hence the measure of its support is at most $\frac 34$ that of the support of $R_{\alpha}$. Hence we have written $R_{\alpha}= fg$, where $f$ and $g$ have a support whose measure is at most $\frac 34$ that of the support of $R_{\alpha}$. Since they are elements of $\iet$, we can again decompose them in restricted rotations, which can be similarly decomposed. By iteration, we can write a rotation as a finite product of elements with arbitrarily small support.
\end{proof}

\begin{lemma} For any $\epsilon >0$, any distinguished involution can be written as the product of elements of $\ietf$ with support of measure less than $\epsilon$. 
\end{lemma}

\begin{proof}
It is enough to prove it for the involution $\mathcal I: x\mapsto 1-x$ on $[0,1)$. Let $n$ be such that $\frac 1n<\epsilon$, and let $f_i$ be such that $f_i(x)=1-x$ if $x\in [\frac i{2n}, \frac{i+1}{2n})\cup(1-\frac {i+1}{2n}, 1-\frac i{2n}]$, and $f_i(x)=x$ otherwise. It is clear that all the $f_i$ have support of measure $\frac 1n<\epsilon$, and by construction $\mathcal I=f_0f_1\ldots f_{n-1}$. 
\end{proof}

\begin{lemma}\label{decomp} For any $\epsilon >0$, any homeomorphism (resp. PL homeomorphism) of $[0,1)$ can be written as the product of homeomorphisms (resp. PL homeomorphisms) whose support are contained in intervals of measure less than $\epsilon$.
\end{lemma}

\begin{proof} We do the proof for a homeomorphism, it works, {\em mutatis mutandis}, for a PL homeomorphism.

Let $h$ be such a homeomorphism; without loss of generality, we can suppose that $h(\frac 12)<\frac 12$. One can then construct a homeomorphism $g$ with support in $[0,\frac 34)$ such that $g(h(\frac 12))=\frac 12$. The homeomorphism $gh$ fixes the point $\frac 12$; hence it can be naturally decomposed in a product $f_1f_2$, where $\supp(f_1)\subset [0,\frac 12]$ and $\supp(f_2)\subset [\frac 12, 1]$. We have written $h=g^{-1}f_1f_2$ as the product of 3 elements whose supports have measure at most $\frac 34$ that of $h$. By iterating this construction, we can make the support of the maps contained in intervals as small as we want.
\end{proof}

If $h$ is a transformation in any of the groups $\iet$, $\ietf$ and  $\ieta$ which is not the identity, we can find an interval $E$ which is disjoint from $h(E)$. Let $\epsilon$ be less than half the length of this interval. Since the support of an element of $\iet$, $\ietf$ and $\ieta$ is a finite union of intervals, if one has two elements $g_1, g_2$ with support of measure less than $\epsilon$, it is clear that we can find an element of $\iet$ which sends the supports of $g_1, g_2$ into $E$.

All these lemmas imply the following : 

\begin{proposition} The groups $\iet$, $\ietf$ and  $\ieta$ satisfy the conditions of Lemma \ref{lemmepst}.
\end{proposition}

Things are slightly more complicated for the group $\ietg$, since the support of a homeomorphism does not need to be a finite union of intervals. However, the reader will check that the proof of Lemma \ref{lemmepst} is still valid if we  reformulate condition (1), by asking the support to be contained in a finite union of intervals with total measure less than $\epsilon$, and change accordingly the condition (2). This is precisely the condition proved in Lemma \ref{decomp}. Hence the group $\ietg$ also satisfies the conclusion of Lemma \ref{lemmepst}.

\subsection{Commutators in groups of interval exchange transformations}
We now want to prove that specific elements are commutators.  
\begin{lemma} Distinguished involutions and restricted rotations are commutators in $\ietf$.\end{lemma}

\begin{lemma} \label{IRCom} Let $\alpha,\beta \in J$.

The maps $\mathcal I_{[\alpha,\beta)}$ and $R_{\theta,[\alpha,\beta)}$ are commutators in $\mathcal G\big([\alpha,\beta)\big)$ and then in $\mathcal G(J)$.
\end{lemma}

\begin{proof} Conjugating by a homothecy, it is sufficient to prove that $\mathcal I_{[0,1)}$ and $R_{\theta,[0,1)}$ are commutators in $\mathcal G([0,1))$ and it is easy to see that $\mathcal I_{[0,1)}$ is the product of the involutions $f_1$ and $f_2$ whose best representatives are described as below: 

{\unitlength=0,3 mm
 \hskip 2.5 truecm \begin{picture}(120,120)
\put (20,98){$\widehat{f_1}$}
\put(0,10){\textcolor{gray}{\line(1,0){100}}} 
\put(10,0){\textcolor{gray}{\line(0,1){100}}}
\put(0,0){$\scriptstyle 0$}
\put(30, 8){$\scriptstyle \vert$} \put(26,0){$\scriptstyle \frac{1}{4}$}
\put(70, 8){$\scriptstyle \vert$} \put(71,0){$\scriptstyle \frac{3}{4}$}
\put(90, 8){$\scriptstyle \vert$} \put(91,0){$\scriptstyle 1$}
\put(9,9){\textcolor{red}{$\scriptstyle  \bullet$}}
\put(29,29){\textcolor{red}{$\scriptstyle  \bullet$}}
\put(69,69){\textcolor{red}{$\scriptstyle  \bullet$}}
\put(10,10){\line(1,1){20}}
\put(30,70){\line(1,-1){40}}
\put(70,70){\line(1,1){20}}
\end{picture} \hskip 3 truecm
\begin{picture}(120,120) 
\put (20,98){$\widehat{f_2}$}
\put(0,10){\textcolor{gray}{\line(1,0){100}}}
\put(10,0){\textcolor{gray}{\line(0,1){100}}}

\put(0,0){$\scriptstyle 0$}
\put(30, 8){$\scriptstyle \vert$}
\put(26,0){$\scriptstyle \frac{1}{4}$}
\put(70, 8){$\scriptstyle \vert$}
\put(71,0){$\scriptstyle \frac{3}{4}$}
\put(90, 8){$\scriptstyle \vert$} \put(91,0){$\scriptstyle 1$}
\put(9,69){\textcolor{red}{$\scriptstyle  \bullet$}}
\put(29,29){\textcolor{red}{$\scriptstyle  \bullet$}}
\put(69,9){\textcolor{red}{$\scriptstyle  \bullet$}}
\put(10,90){\line(1,-1){20}}
\put(30,30){\line(1,1){40}}
\put(70,30){\line(1,-1){20}}
\end{picture}}

\smallskip

As $f_2$ is conjugated to $f_1=f_1^{-1}$ by $r=R_{\frac{1}{2}}$, one has $\mathcal I_{[0,1)} =f_1 r f_1^{-1}r^{-1}$ is a commutator.

\smallskip

In addition, according to Remark \ref{Stheta}, any rotation is the product of $2$ involutions that are conjugated by a rotation; thus  $R_{\theta,[0,1)}$ is a commutator. 
\end{proof}

Since any element of $\ietf$ is a product of a distinguished involution and restricted rotations, this implies that any element of $\ietf$ is a product of commutators; and since we have proved that the commutator subgroup is the smallest normal subgroup of $\ietf$, we have proved
\begin{proposition}
The group $\ietf$ is simple.
\end{proposition}

Let us now consider the group $\ieta$. Conjugating by a homothecy, it is sufficient to consider the group $\ietaa$.
\begin{lemma}
Every element of $PL^+([0,1])$ is a product of commutators.
\end{lemma}
\begin{proof} This result is proved in \cite{Ep}; for completeness, we give the main point of the argument. Any piecewise affine homeomorphism can be written as a product of maps which are the identity out of an interval $[a,b]$, and which are affine on $2$ intervals $[a,c]$ and $[c,b]$. It suffices to write such a map as a commutator. 

Denote by $\sigma_a$ the piecewise affine homeomorphism of $[0,1]$ which fixes 0 and 1, sends $\frac 12 $ to $a$, and is affine on $[0,\frac 12]$ and $[\frac 12, 1]$. If we choose $a,b$ with $0<a<b<\frac 12$, it is easily checked that the commutator $\sigma_a^{-1}\sigma_b^{-1}\sigma_a\sigma_b$ is a PL map which is the identity out of the interval $\frac 12, \frac {3-4a}{4-4a}$, and which is linear on $2$ intervals. 
A simple study shows that, up to conjugacy, one obtains in this way any piecewise affine map on $2$ intervals.
\end{proof}

\begin{lemma}
Any rotation is a product of commutator in $\ietaa$
\end{lemma}
\begin{proof}
The group $PL^+(\mathbb S^1)$ of piecewise affine homeomorphisms of the circle  can be embedded in $\ietaa$ as piecewise affine transformations of the interval $[0,1)$. But (see \cite{Ep}) this group is simple; hence any rotation of $[0,1)$ can be written as a product of commutators. 
\end{proof}

Hence any element of $\ieta$ is a product of commutators, and we prove as above: 
\begin{proposition}
The group $\ieta$ is simple.
\end{proposition}

\begin{remark}
The proof that rotations are product of commutators is fundamentally different in $\ieta$ and $\ietf$; and indeed, the property is false in their intersection $\iet$.
\end{remark}

We end with the proof of the simplicity of $\ietg$ reduced to that of $\ietgg$.

\begin{lemma}
Any rotation and any orientation preserving homeomorphism of $[0,1]$ is a product of commutators. 
\end{lemma}
\begin{proof}
The proof given for rotations in $\ietaa$ is still valid in $\ietgg$, since this group contains $\ietaa$.

We proved above that any homeomorphism of $[0,1]$ is a product of homeomorphisms whose supports are contained in small intervals; conjugating by a rotation, we can consider a homeomorphism $h$ of $[0,1]$ whose support is included in $[\frac 14, \frac 12]$.

Define $\phi$ on $[0,1]$ by $\phi(x)=\frac x2$ if $x<\frac 12$, $\phi(x)=\frac {3x}2-\frac 12$ if $x>\frac 12$. We define a sequence of functions $f_i$ by $f_1=\phi h\phi^{-1}$, $f_{i+1}=\phi f_i\phi^{-1}$ for $i>1$. The sequence $f_i$ converges uniformly to the identity, and they have disjoint support; hence the sequence $g_n=f_1f_2\ldots f_n$ converges to a function $g$ which is a homeomorphism of $[0,1]$ and verifies $\phi^{-1}g\phi=hg$; hence $h$ is a commutator.
\end{proof}
As above, this proves that $\ietg$ is simple.

It is well-known that the group $\iet $ is not simple, and not equal to its commutator subgroup, this is provided by the following
\begin{theo}(Arnoux-Fathi-Sah 1981 \cite{Ar}) Let $\lambda$ be the Lebesgue measure on $[0,1)$, $$\mbox{the application } \saf : \left\{ \begin{array}{ll} \iet \rightarrow \R \otimes_\Q \R \cr \quad f \quad \mapsto \saf(f) = {\ds \sum_{\alpha\in \R}} \alpha \otimes  \lambda((f-\id)^{-1} (\{\alpha\}))\end{array} \right.$$ is a morphism and its kernel is the commutator subgroup of $[\iet,\iet]$.
\end{theo}
But we have:  
\begin{proposition}
The group $[\iet, \iet]$ is simple.
\end{proposition}

\begin{proof} \ 

The group $\iet$ satisfies the conditions of Lemma \ref{lemmepst}, which implies that its group of commutator is the smallest normal subgroup. Since the commutator subgroup of $[\iet,\iet]$ is normal in $\iet$, the group $[\iet, \iet]$ is perfect. 

It remains to prove that the commutator subgroup satisfies also the conditions of Lemma \ref{lemmepst}. Let $f\in [\iet,\iet]$.  The $\saf$-invariant of an involution, being equal to its own opposite, is zero so any involution is a product of commutators, and eventually composing with an involution which exchanges a small interval and its image by $f$, we can assume that $\supp(f) \not=[0,1)$. By section A.2, $f$ can be decomposed as  $f=g_1\ldots g_n$, a product of interval exchange transformations with small support included in $\supp(f)$. There is no reason for $g_i$ to be a product of commutators; but in that case, its invariant $\saf(g_i)$ is not $0$, and we can find maps $h_i$ with small support disjoint from the supports of the $g_i$ such that $\saf(h_i) = \saf(g_i)$; hence $h_i$ commute with all the $g_k$, and we can write: $$f=g_1h_1^{-1}\ldots g_n h_n^{-1} h_n\ldots h_1.$$
Since $f$ is a product of commutators, $\saf(f)=0$. By construction, $\saf(g_ih_i^{-1})=0$, hence it is a product of commutator; if we define $k=h_n\ldots h_1$, we have, taking the invariant of both sides, $\saf(k)=0$, hence $k$ is a product of commutator, and $f=(g_1h_1^{-1})\ldots (g_n h_n^{-1})k$ is a decomposition in product of commutators with arbitrarily small size. 

This proves the first condition; to prove the second condition we can find involutions sending a finite union of intervals inside an interval of larger measure.
\end{proof}
\end{appendix}

\bibliographystyle{alpha} \bibliography{RefUSimple}

\end{document}